\documentclass[final]{siamart0516}

\usepackage{lipsum}
\usepackage{amsfonts}
\usepackage{graphicx}
\usepackage{algorithmic}
\usepackage{epstopdf}

\usepackage{array,makecell}
\usepackage{amssymb}
\usepackage{upgreek}
\usepackage{bm}
\usepackage{float}
\usepackage{amsmath}
\usepackage{multicol}
\usepackage{colortbl} 
\usepackage{caption}
\usepackage{hyperref}
\usepackage{multirow}
\usepackage{hhline}

\usepackage{cases}

\ifpdf
  \DeclareGraphicsExtensions{.eps,.pdf,.png,.jpg}
\else
  \DeclareGraphicsExtensions{.eps}
\fi

\makeatletter
\newsavebox\myboxA
\newsavebox\myboxB
\newlength\mylenA

\newcommand*\xoverline[2][0.75]{%
    \sbox{\myboxA}{$\m@th#2$}%
    \setbox\myboxB\null
    \ht\myboxB=\ht\myboxA%
    \dp\myboxB=\dp\myboxA%
    \wd\myboxB=#1\wd\myboxA
    \sbox\myboxB{$\m@th\overline{\copy\myboxB}$}
    \setlength\mylenA{\the\wd\myboxA}
    \addtolength\mylenA{-\the\wd\myboxB}%
    \ifdim\wd\myboxB<\wd\myboxA%
       \rlap{\hskip 0.5\mylenA\usebox\myboxB}{\usebox\myboxA}%
    \else
        \hskip -0.5\mylenA\rlap{\usebox\myboxA}{\hskip 0.5\mylenA\usebox\myboxB}%
    \fi}
\makeatother

\usepackage{tikz}

\usetikzlibrary{shapes,arrows,positioning}
\definecolor{dark_blue}{RGB}{46,87,144}
\definecolor{blue_pers}{RGB}{54,104,171}
\definecolor{grey_pers}{RGB}{245,245,245}
\definecolor{grey_ref}{RGB}{201,201,201}
\definecolor{red_pers}{RGB}{213,78,33}
\definecolor{green_pers}{RGB}{194, 239, 194}

\usepackage{xcolor}

\newcommand\revn[1]{\iffalse #1 \fi}

\tikzstyle{raw_text} = [rectangle, minimum height=4em]
\tikzstyle{box_text} = [rectangle, text width=2.4em, 
    text centered, minimum height=4em]

\tikzstyle{block_par} = [rectangle, draw, fill=grey_pers!30, 
    text width=6em, minimum height=2.2em]
    \tikzstyle{block_par_var} = [rectangle, draw, fill=grey_pers!30, 
    text width=6.5em, minimum height=2.2em]
\tikzstyle{block} = [rectangle, draw, fill=grey_pers!30, 
    text width=8em, text centered, rounded corners, minimum height=3.5em]

\tikzstyle{type} = [rectangle, draw, fill=grey_pers, 
    text width=8em, text centered, minimum height=3.5em]    
\tikzstyle{objection} = [rectangle, draw, 
    text width=8em, text centered, rounded corners, minimum height=3.5em]
\tikzstyle{ref} = [rectangle, draw, fill=grey_pers, 
    text width=3em, text centered, minimum height=3.5em]    
\tikzstyle{ref1} = [rectangle, draw, 
    text width=3em, text centered, rounded corners, minimum height=3.5em]

\tikzstyle{line} = [draw, -latex']
\tikzstyle{line_1} = [draw]
\tikzstyle{cloud} = [draw, ellipse,fill=grey_pers, node distance=3cm,
    minimum height=3em]
\tikzstyle{circle_blue} = [draw,circle,fill=grey_ref, node distance=2cm,
    minimum size=2em]
    \tikzstyle{circle_green} = [draw,circle,fill=green_pers,
     node distance=2cm,
    minimum size=2em]
\tikzstyle{circle_grey} = [draw,circle,fill=grey_pers, node distance=2cm,
    minimum size=2.2em] 
\tikzstyle{circle_red} = [draw,circle,fill=red_pers!30, node distance=2cm,
    minimum size=2em] 
\tikzstyle{back} = [rectangle, draw, 
    text width=5em, text centered, rounded corners, minimum height=1.5em]

\tikzstyle{back_b} = [rectangle, draw, fill=blue_pers!30, 
    text width=5em, text centered, rounded corners, minimum height=1.5em]

\tikzstyle{back_g} = [rectangle, draw, fill=grey_pers!30, 
    text width=5em, text centered, rounded corners, minimum height=1.5em]

\usepackage{color}

\newcommand{\half}{\frac{1}{2}}

\newcommand{\ba}{{\bf a}}
\newcommand{\boldbeta}{{\boldsymbol{\beta}}}

\newcommand{\btheta}{{\boldsymbol{\theta}}}

\newcommand{\bb}{{\bf b}}

\newcommand{\bh}{\boldsymbol{h}}
\newcommand{\bxi}{\boldsymbol{\xi}}

\newcommand{\bn}{{\bf n}}

\newcommand{\bx}{{\bf x}}

\newcommand{\bv}{{\bf v}}

\newcommand{\bI}{{\bf I}}

\newcommand{\U}{{\textup{U}}}

\newcommand{\D}{\mathsf{d}}
\newcommand{\fr}{\mathsf{r}}
\newcommand{\fq}{\mathsf{q}}
\newcommand{\fk}{\mathsf{k}}
\newcommand{\fs}{\mathsf{s}}

\newcommand{\fK}{\mathsf{K}}
\newcommand{\fH}{\mathsf{H}}

\newcommand{\fX}{\mathsf{X}}
\newcommand{\fY}{\mathsf{Y}}
\newcommand{\fZ}{\mathsf{Z}}
\newcommand{\fD}{\mathsf{D}}
\newcommand{\fC}{\mathsf{C}}

\newcommand{\IN}{\mathbb{N}}
\newcommand{\IR}{\mathbb{R}}

\newcommand{\IH}{\mathbb{H}}

\newcommand{\loc}{{\textup{loc}}}

\newcommand{\divg}{\operatorname{div}}

\newcommand{\be}{\begin{equation}}
\newcommand{\ee}{\end{equation}}
\newcommand{\norm}[2]{\left\|#1\right\|_{#2}}

\newcommand{\Hloc}{H_{\textup{loc}}}

\newcommand{\eps}{\varepsilon}

\newcommand{\SRC}{\textup{SRC}}

\newcommand{\mB}{\mathcal{B}}

\newcommand{\mE}{\mathcal{E}}

\newcommand{\mH}{\mathcal{H}}

\newcommand{\mO}{\mathcal{O}}

\newcommand{\mfT}{\mathfrak{T}}

\usepackage{array,makecell}
\usepackage{multirow}
\usepackage{subcaption}

\usepackage{tikz}

\definecolor{grey_pers}{RGB}{245,245,245}
\tikzstyle{write_text} = [rectangle, 
    text width=2em, text centered, minimum height=4em]
    \tikzstyle{block} = [rectangle, draw, fill=grey_pers!30, 
    text width=3em, text centered, minimum height=3em]
\tikzstyle{line} = [draw]

\usepackage[normalem]{ulem}
\definecolor{dark_green}{RGB}{0,100,0}
\definecolor{dark_blue}{RGB}{46,87,144}

\newcommand{\TheTitle}{Shape differentiability of Helmholtz scattering problems via explicit shape calculus}
\newcommand{\TheAuthors}{Paul Escapil-Inchausp\'e and Carlos Jerez-Hanckes}

\headers{Shape differentiability of Helmholtz scattering problems}{P.~Escapil-Inchausp\'e and C.~Jerez-Hanckes}

\title{{\TheTitle}\thanks{Submitted to the editors \today.
\funding{This work was supported in part by Fondecyt Regular 1171491.}}
}
\author{
  Paul Escapil-Inchausp\'{e}\thanks{School of Engineering, Pontificia Universidad Cat\'olica de Chile, Santiago, Chile (\email{pescapil@uc.cl}).}
  \and
 Carlos Jerez-Hanckes\thanks{Faculty of Engineering and Sciences, Universidad Adolfo Ib\'a\~nez, Santiago, Chile (\email{carlos.jerez@uai.cl}).}}
 
\usepackage{amsopn}

\usepackage{amsfonts} 

\newsiamremark{remark}{Remark} 
\newsiamthm{problem}{Generic Problem}
\newsiamthm{subproblem}{Problem}

\ifpdf
\hypersetup{
  pdftitle={\TheTitle},
  pdfauthor={\TheAuthors}}

\begin{document}
\date{\today}
\maketitle

\begin{abstract}
We consider the time-harmonic scalar wave scattering problems with Dirichlet, Neumann, impedance and transmission boundary conditions. Under this setting, we analyze how sensitive diffracted fields and Cauchy data are to small perturbations of a given nominal shape. To this end, we follow [K.~Eppler, \emph{Int.~J.~Appl.~Math.~Comput.~Sci.}~10(3) (2000), pp.~487-516], referred to as explicit shape calculus, and which places great emphasis on the characterization of the domain derivatives as boundary value problems. It consists in combining elliptic regularity theorems, shape calculus and functional analysis, allowing to deduce sharp differentiability results for the domain-to-solution and domain-to-Cauchy data mappings. The technique is applied to both classic and limit Sobolev regularity cases, leading to new and comprehensive differentiability results. 
\end{abstract}




\section{Introduction}\label{sec:intro}
Modern scientific and technological progress strongly rely on mastering wave propagation in unbounded domains. Areas ranging from biomedical imaging to remote sensing and space exploration greatly profit from the ever-increasing computational processing capacity to simulate wave scattering. The boundary value problems (BVPs) considered here involve solving the Helmholtz equation in unbounded domains supplemented by one or more different boundary conditions (BCs), namely, Dirichlet, Neumann, impedance and transmission ones.

Questions such as how do quantities of interest vary under a suitably defined perturbation from a given \emph{nominal domain} or how regular these perturbations are constitute the center of attention of shape sensitivity analysis. Besides the obvious applications of this analysis in optimization, one sees it in inverse problems ({\emph{cf}.~\cite[Chapter 4 and 5]{colton2012inverse} and \cite{hettlich1995frechet,hettlich1998frechet,Frechet,potthast1996frechet}), first-order second statistical moment approximation for uncertainty quantification \cite{dolz2018hierarchical,escapil2020helmholtz}, and even shape holomorphy analysis \cite{jerez2017electromagnetic}. Appealing to \emph{bi-Lipschitz diffeomorphisms} between nominal and perturbed domains, one carries out \emph{shape calculus}---see \cite{Allaire,antoine_var,Sokolowski} for an exhaustive overview of these techniques. Besides, \cite[Chapter 9]{delfour2011shapes} presents a comprehensive classification of shape differentiability properties---directional, G\^ ateaux and Fr\'echet derivative (resp.~differentiability).

Under domain derivatives (DDs) we group two closely related concepts in shape calculus: the \emph{material derivative} (MD) and  \emph{shape derivative} (SD). The former is generally defined as a classical G\^ateaux or Fr\'echet derivative in Banach spaces on a fixed domain, customarily allowing for regularity results related to the \emph{implicit function theorem} (see, e.g.,~\cite{Sokolowski,antoine_var}). As the latter has varying supports, its analysis is more involved being commonly built upon the MD, though prone to regularity loss due to the arising \emph{Lie derivative} (LD) \cite{hiptmair2017shape,costabel2012shape}. Furthermore, both DDs can be characterized as BVPs. Still, the SD has a simpler characterization as the solution of a BVP with \emph{Cauchy data} depending on the normal component of the velocity field, allowed by the Hadamard structure theorem (\emph{cf}.~\cite[Theorem 2.27]{Sokolowski}). The connection between both DDs can make their characterization confuse, and justifies their definition in a rigorous setting. Also, notice that it can be advantageous to use the so-called LD-based configuration, which consists in resorting to the MD to characterize the SD \cite{hiptmair2017shape} and conversely \cite{hagemann2019solving}.

To our knowledge, the SD for Helmholtz wave scattering was first computed by Kirsch \cite{kirsch1993domain}. Recently, a general framework to consider the DDs based on differential forms was presented by Hiptmair and Li \cite{hiptmair2013shape,hiptmair2017shape}, though explicit BVPs for the MDs were not provided. For less regular DDs, Bochniak and Cakoni considered DDs in nonsmooth domains \cite{bochniak2002domain} and Kleeman extended the work to nonsmooth perturbations \cite{kleemann2012shape}. Notwithstanding these remarks, and despite being critical topics for scattering problems, a precise characterization of the DDs' regularity with respect to the domain, data and transformation is still missing.

To summarize, the general procedure to prove differentiability results is achieved through application of the implicit function theorem to the MDs \cite{djellouli1999continuous,barucq2017frechet}. Rather independently, a characterization step follows ---characterization of the SDs as BVPs \cite{djellouli1999characterization,barucq2018mathematical}. In addition, there seems to be no available description  of the MDs as BVPs---aside from \cite{bochniak2002domain} for mixed scattering problems and the rather abstract framework \cite[Equation 4.3]{hiptmair2017shape}---nor of the DDs of the associated Cauchy data, referred to as CMD and CSD for material and shape derivative, respectively. 

In this work, we prove several differentiability and stability results concerning the domain-to-solution and domain-to-Cauchy data for a family of Helmholtz scattering problems. This is achieved via the \emph{explicit shape calculus}, a procedure followed by Eppler in \cite{eppleroptimal} which consists in two steps:

\begin{enumerate}
\item[(S1)] Characterization: describe explicitly the directional DDs throughout their BVPs;
 \item[(S2)] Construction: ensure a Banach space embedding to prove G\^ateaux differentiability, and deduce continuous Fr\'echet differentiability.
 \end{enumerate}
Notice that the core of explicit shape calculus is (S1), as the characterization is often tedious and requires to be as sharp as possible. This greatly improves the technique, as the analysis of BVP is a well-developed branch in mathematics. Well-posedness of the BVPs for the DDs---in particular their $H^1$-regularity---is key to study detailedly the continuous data dependence of quantities of interest.

Regarding (S2), it follows by simple arguments provided functional spaces are set such that the Banach space embedding is straightforward. This technique allows for a fine characterization of the differentiability properties, as it encloses (S1) to the simple concept of directional derivative, allowing for a general framework and adaptability. Besides, the BVPs keep track of the constants involved in the DDs and allow for a straightforward analysis. Further advantages of this procedure are
\begin{enumerate}
 \item[(i)] it is applicable to any type of shift/regularity theorem;
 \item[(ii)] it considers sharp regularity results in the limit case for smooth data, which is of interest for time-harmonic wave scattering (see \cite{costabel1988boundary,jerison1981dirichlet,jerison1981neumann,costabel2019limit,costabel1990remark} for Laplace case over Lipschitz domains);
 \item[(iii)] it is suitable for wavenumber analysis \cite{spence2014wavenumber,moiola2019acoustic,sauter2018stability,chaumont2018wavenumber} and paves the way towards existence results for the SD for Lipschitz domains and domains with corners.
\end{enumerate}
For this purpose, we split both domain and transformation smoothness and provide sharp regularity estimates for each DD respect to those parameters. Accordingly, we characterize and describe accurately the behavior of the CDDs.

For the sake of understanding, a comprehensive illustration is given in \cref{fig:Representation}, including depictions of the domains and acronyms used throughout. Along with this, a black-box representation of the regularity considerations for (S1) is provided in \cref{fig:GD}.

\begin{figure}[t]
 \begin{minipage}{0.25\linewidth}
\begin{figure}[H]
\centering
\includegraphics[width=1\textwidth]{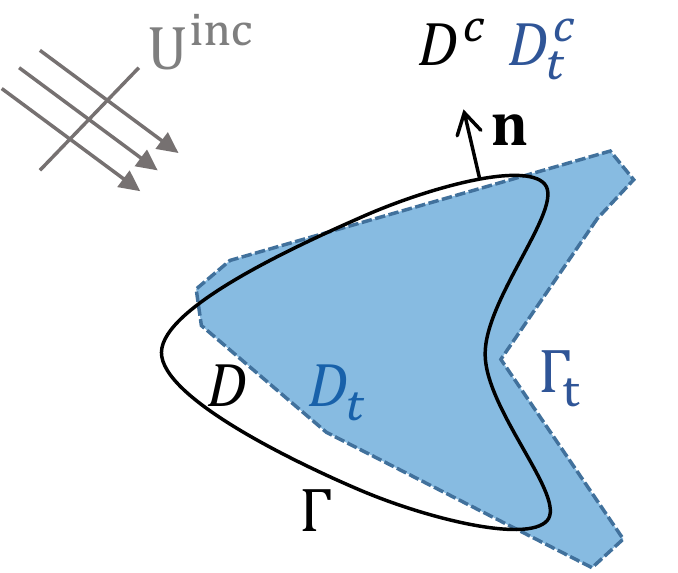}
\end{figure}
\end{minipage}
\begin{minipage}{0.44\linewidth}
\centering
\resizebox{5cm}{!} {
\begin{tabular}{|c|c|c|} \hline   
BVP & Boundary Value Problem \\ \hline
BC & Boundary Condition \\ \hline
DD & Domain Derivative \\ \hline
MD & Material Derivative\\ \hline
SD & Shape Derivative\\ \hline
LD & Lie Derivative\\ \hline
C $\cdot$ & $\cdot$ of Cauchy Data\\ \hline
\end{tabular}}
\end{minipage}
\begin{minipage}{0.25\linewidth}
\begin{figure}[H]
\includegraphics[width=1\textwidth]{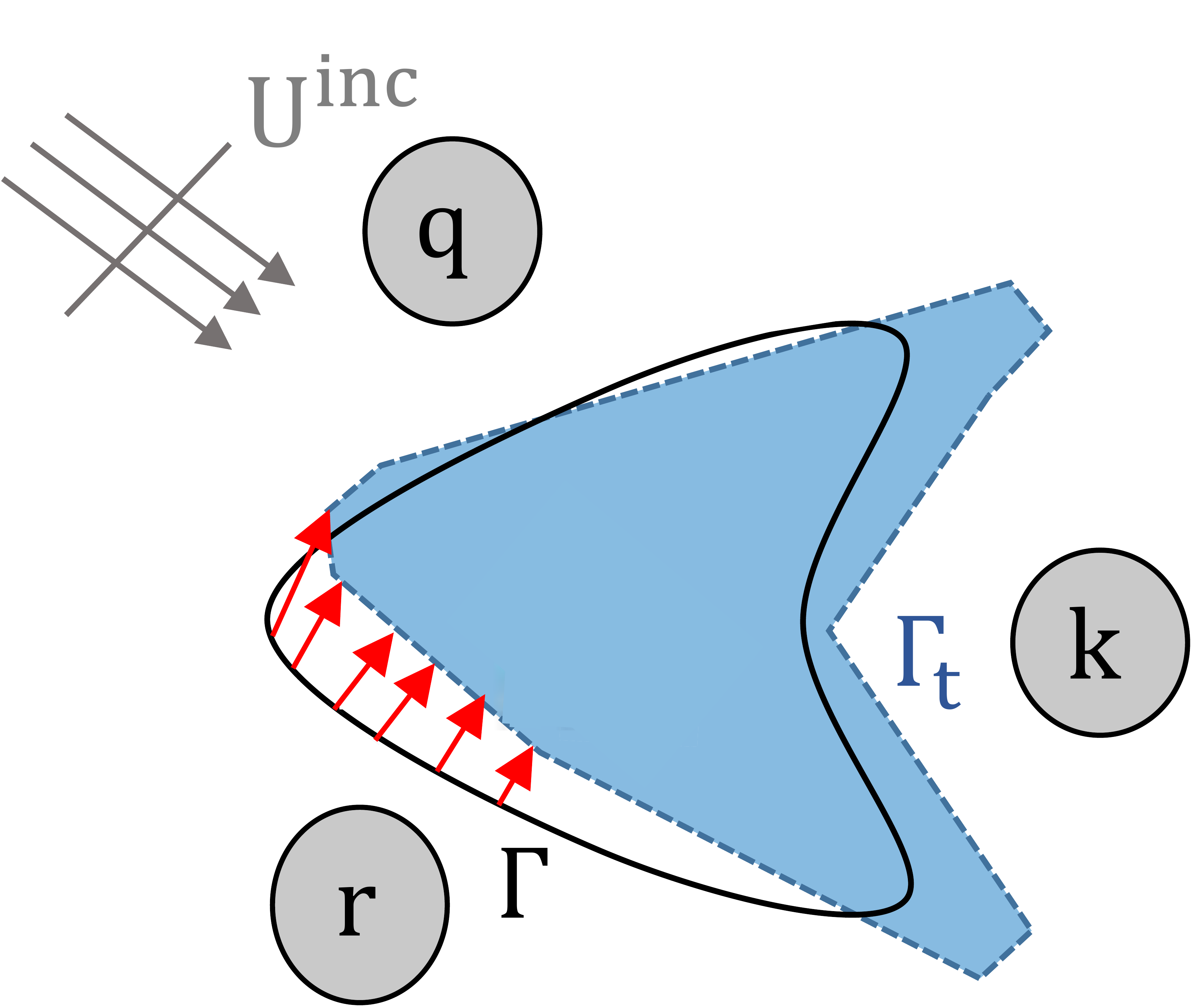}
\label{fig:fig2}
\end{figure}
\end{minipage}
\caption{Representations of domain transformation encircling a list of the important acronyms used throughout this work.}
\label{fig:Representation}
\end{figure}

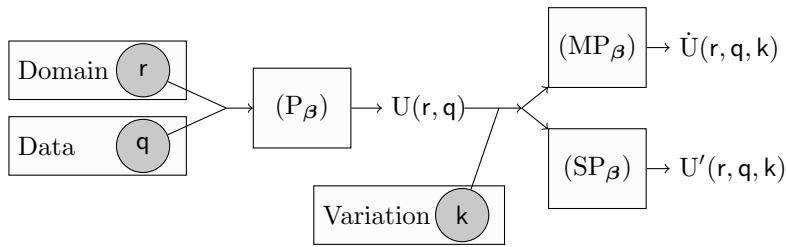
\begin{figure}[t]
\centering
\resizebox{10cm}{!} {
 \begin{tikzpicture}[node distance = 2cm, auto]
    \node [box_text] (init)  {};
    \node [block_par, above of = init, node distance =.5cm] (iii)  {Domain};
    \node [circle_blue, right of = iii, node distance = 0.6cm] (init1)  {$\fr$};
    \node [box_text,right of = init1, node distance=0.2cm] (i1) {} ;
    \node [box_text,right of = i1, node distance=1.5cm] (i11) {} ;

    \node [block_par, below of = init, node distance =.5cm] (iii2)  {Data};
    \node [circle_blue, right of= iii2, node distance =0.6cm] (init2)  {$\fq$};

    \node [box_text,right of = init2, node distance =0.3cm] (i2) {};
    \node [box_text,right of = i2, node distance =1.4cm] (i21) {};
  
   \node [block, right of =init, node distance = 2.7cm ] (func) {(P$_\boldbeta$)};
   \coordinate[ left of = func, node distance = 1cm] (a1);

   \node [box_text, right of =func, node distance =1.6cm] (init3) {$\U(\fr,\fq)$};

   \coordinate[ right of =init3, node distance = 1cm] (a2);
   \coordinate[ right of =a2, node distance = .3cm] (a3);

   \coordinate[ right of =a3, node distance = 1cm] (a4);

   \node [block_par_var, below right of = func, node distance = 2cm] (iii3)  {Variation};

   \node [circle_blue,right of =iii3, node distance =0.7cm] (init31) {$\fk$};

  \node [block, above of=a4, node distance =.8cm ] (funcMP) {(MP$_\boldbeta$)};
  \node [block, below of =a4, node distance =.8cm ] (funcSP) {(SP$_\boldbeta$)};
  
  \node [box_text,right of =funcMP, node distance =1.5cm] (finMP) {$\dot{\U}(\fr,\fq,\fk)$};
  \node [box_text,right of =funcSP, node distance =1.5cm] (finSP) {$\U'(\fr,\fq,\fk)$};
   \draw [-] (init1)--(a1) ;
   \draw [-] (init2)--(a1) ;
   \draw [->] (a1)--(func) ;
   \draw [->] (func)--(init3) ;
   \draw [-] (init3) -- (a2);
   \draw [-] (init31) -- (a2);
   \draw [->] (a2) -- (a3);
   \draw [->] (a3) -- (funcMP);
   \draw [->] (a3) -- (funcSP);
   \draw [->] (funcMP) -- (finMP);
   \draw [->] (funcSP) -- (finSP);
\end{tikzpicture}}
\caption{A diagram representation of the parametrized study of this manuscript in (S1) for a QoI $\U$ function to smoothness indices representing the initial domain $\fr$, the data for the problem $\fq$ and the domain variation $\fk$. The final output parameter are the DDs of $\U$. Notice that both DDs are function to $\fr,\fq,\fk$. }
\label{fig:GD}
\end{figure}

This manuscript is structured in the following way. First, we introduce the mathematical tools required in \cref{sec:mathematical_tools}, and we present our main results in \cref{sec:main_results}. The rest of this work is devoted to proving those results. To this end, we set a framework for domain sensitivity and describe the BVPs for the DDs in \cref{sec:Shape_calculus}. Afterwards, we apply twice the characterization step (S1) to both classic and limit shift theorems in \cref{sec:regularity} and end up the explicit shape calculus --and the formal proof of the results in \cref{sec:main_results}-- with \cref{sec:Banach}. Further research avenues are presented in \cref{sec:concl}.

\section{Mathematical tools}
\label{sec:mathematical_tools}
We set the basic definitions used all throughout in order to state the main results of this work.
\subsection{General notation}
\label{subs:prelim}
Throughout, vectors and matrices are expressed using bold symbols, $\ba \cdot \bb$ denotes the classical Euclidean inner product, $\norm{\cdot}{2}:=\sqrt{\ba \cdot \ba}$ refers to the Euclidean norm, $C$ is a generic positive constant and $o$, $\mO$ are respectively the usual little-$o$ and big-$\mO$ notations. Also, we set $\imath^2=-1$.

Let $D \subseteq  \IR^d$, with $d=2,3$, be an open set. By $B_R$ we denote a ball of radius $R>0$ centered at zero. For $p \in \IN_0:=\IN\cup\{0\}$, we denote by $C^{p}(D)$ the space of $p$-times differentiable functions over $D$, $C^{p,\alpha}(D)$ is the space of H\"older continuous functions with exponent $\alpha$, where $0<\alpha\leq1$. Also, let $L^p(D)$ be the standard class of functions with bounded $L^p$-norm over $D$. For a Banach space $X$ and an open set $T \subset \IR$, we introduce the usual Bochner space $C^p(T;X)$. For $s \in \IR$, $q\geq 0$, $p\in [1,\infty]$, we refer to \cite[Chapter 3]{MacLean} for the definitions of functional spaces $W^{s,p}(D)$, $H^s(D)$, $\Hloc^q(D)$ and $\Hloc^q(\Delta,D)$. Norms are denoted by $\norm{\cdot}{}$, with subscripts indicating the associated functional space.

\subsection{Lipschitz domains}\label{subs:Lipschitz}
For $r\in \IN_0$, notation $\Gamma \in C^{r,1}$, refers to an open bounded Lipschitz domain $D \subset \IR^d$, $d=2,3$, of class $C^{r,1}$ \cite[Definition~1.2.1.2]{Grisvard}, with boundary $\Gamma:=\partial D$. The exterior unit normal field $\bn\in W^{r,\infty}(\Gamma)$ is pointing by convention towards the exterior domain. Finally, the mean curvature reads $\mathfrak{H}:=\divg \bn$ and belongs to $W^{r-1,\infty}(\Gamma)$. 

\subsection{Traces and surface operators}
Let $D  \subset \IR^d$, with $d=2,3$ be an open bounded Lipschitz domain with boundary $\Gamma$ and complementary exterior domain $D^c:=\mathbb{R}^d \backslash \overline{D}$. Equivalently, we will write $D^0 \equiv D^c$ and $D^1\equiv D$ to refer to exterior and interior domains, respectively. Accordingly, when defining scalar fields in $D^c\cup D$, we use the shorthand $\U=(\U^0,\U^1)$. Let us introduce the continuous and surjective trace mappings \cite[Lemma 3.2]{MacLean}:
\begin{equation*}
\begin{array}{cl}
\textup{(Dirichlet trace)} & \gamma_0 : \Hloc^{1}(D^i) \to  H^{\frac{1}{2}}(\Gamma),\\
\noalign{\vspace{2pt}}
\textup{(Neumann trace)}   & \gamma_1 : \Hloc^{1}(\Delta,D^i) \to  H^{-\frac{1}{2}}(\Gamma).
\end{array}
\end{equation*}
For a suitable scalar field $\U^i$, $i=0,1$, we refer to a pair of traces $\bxi^i$ as \emph{Cauchy data} if
\begin{equation}
\label{eq:def_cauchy_data}
\bxi^i \equiv \begin{pmatrix}
\xi_0^i\\
\xi_1^i
\end{pmatrix} \equiv 
\begin{pmatrix}
\lambda^i\\
\sigma^i
\end{pmatrix}:=\begin{pmatrix}
\gamma_0\U^i\\
\gamma_1\U^i
\end{pmatrix} = \gamma \U^i,
\end{equation}
with their natural space being the Cartesian product:
\be\IH(\Gamma):=H^\half(\Gamma)\times H^{-\half}(\Gamma).\ee
Likewise, we introduce the second-order trace operator $\gamma_2\U^i := (\nabla^2\U^i |_\Gamma) \bn \cdot \bn  = \frac{\partial^2\U^i}{\partial \bn^2}\big|_\Gamma$. For infinitely smooth domains and functions, for all $s\in \IR$, the tangential gradient $\nabla_\Gamma$ is continuous from $H^{s+1}(\Gamma)$ to $[H^s(\Gamma)]^d$ and the tangential divergence $\divg_\Gamma$ is continuous from $[H^{s+1}(\Gamma)]^d$ to $H^{s}(\Gamma)$ \cite[Section 5]{costabel2012shapeO}. Trace jumps across $\Gamma$ are denoted by
$$
[\gamma_\beta \U]_\Gamma := \gamma_\beta \U^0 - \gamma_\beta \U^1\quad\text{for}\quad\beta \in \{0,1\}. 
$$
\subsection{Functional spaces}
In this work, we consider the Helmholtz problems in two and three dimensions. To begin with, we represent the physical domains by a positive wave speed $c$ and a material density constant $\mu$---representing, for instance, the permeability in electromagnetics. For time-harmonic excitations of angular frequency $\omega>0$, set the wavenumber $\kappa:=\omega/c$ and define the Helmholtz operator:
$$L_\kappa:\U \mapsto -\Delta \U - \kappa^2 \U.$$
For $d=2,3$ and $\kappa>0$, the Sommerfeld radiation condition (SRC) \cite[Section 2.2]{Nedelec} for $\U$ defined over $D^c$  reads
\be\label{eq:SRC}
\SRC(\U,\kappa) \iff \Big|\frac{\partial}{\partial r} \U-\imath \kappa \U\Big|  = o\big( |\bx|^{\frac{1-d}{2}}\big) \text{ for } |\bx| \to \infty.
\ee
This condition will guarantee uniqueness of solutions for the problems described below. Besides, for $q\geq 0$, we introduce the weighted Hilbert space incorporating the behavior at infinity \cite{bochniak2002domain}:
\be\label{eq:Weighted_Sobolev}
\|\U\|_{H^q_\kappa(D^c)} : =  \| \U / (1 + |\bx|^2)^{(d-1)/2}\|_{H^q(D^c)}.
\ee
For any function defined over $D^c$ (resp.~$D^c \cup D$), we use the sans-serif notation $H^q(\fD)$---with $\cdot$ being either $\textup{loc}$ or $\kappa$---as follows:
\be
H^q(\fD) := \left\{
\begin{array}{lll}H_\cdot^q (D^c) &\quad\text{if}\quad \fD =D^c, \\
H^q_\cdot(D^c) \times H^q(D)&\quad\text{if}\quad\fD =D^c \cup D.
\end{array}\right.
\ee
Consequently, $\fH^q(\fD)$ refers to either $H^q_\loc(\fD)$ or $H^q_\kappa(\fD)$, hence the final $\fH^q(\fD)$-notation.
\subsection{Helmholtz Problems Formulations}
\label{subsec:prob_helm_scat}
For an incident wave $\U^\text{inc}$, we introduce the following BVPs corresponding to Dirichlet, Neumann, impedance and transmission BCs. By linearity, we can write the total wave as a sum of scattered and incident terms, i.e.~$\U=\U^\text{sc}+\U^\text{inc}$. Notation (P$_\boldbeta$) with bold $\boldbeta$ will refer to any of the problems considered. Exterior problems (P$_\beta$), $\beta=0,1$, represent the \emph{sound-soft} and \emph{sound-hard} acoustic wave scattering while (P$_2$) and (P$_3$) are  the \emph{exterior impedance} and \emph{transmission} problems. 

Notice that (P$_\boldbeta$) is known to be well posed \cite[Chapter 4]{MacLean}. In \cref{tab:Notations_Problem}, we summarize notations for the different problems. To begin with, we set the BCs. Also, we detail $\fD$ the domain of definition of the BVPs, namely $\fD=D^c$ for $\beta \in \{0,1,2\}$ and $\fD:= D^c\cup D$ for $\beta=3$. To finish, we define boundary data spaces using calligraphic symbols.
\begin{table}[t]
\centering
\renewcommand\arraystretch{1.4}
\resizebox{13cm}{!} {
\begin{tabular}{|c|c|c|c|c|c|c|} \hline
 $\beta$ &  Problem  &  BCs &$\fD$  & \multicolumn{2}{|c|}{$\fH^1(\fD)$}  & $\mH(\Gamma)$\\ \hline\hline
 0 &  Sound-soft & $\gamma_0 \U= 0$ & $D^c$& $H^1_\textup{loc}(D^c)$  & $H_\kappa^1(D^c)$ & $H^\half(\Gamma)$\\ \hline
 1 &  Sound-hard & $\gamma_1 \U= 0$& $D^c$ &  $H^1_\textup{loc}(D^c)$ & $H_\kappa^1(D^c)$ &$H^{-\half}(\Gamma)$\\ \hline
 2 &  Impedance  & $\gamma_1 \U + \imath \eta \gamma_0 \U = 0$ & $D^c$&$H^1_\textup{loc}(D^c)$ &  $H_\kappa^1(D^c)$ &$H^{-\half}(\Gamma)$ \\ \hline
3&  Transmission &  $[\gamma_0\U]_\Gamma  = [\mu^{-1}\gamma_1\U]_\Gamma =0$ & $D^c \cup D$&$H^1_\textup{loc}(D^c)\times H^1(D)$ & $H_{\kappa_0}^1(D^c)\times H^1(D)$ & $H^\half(\Gamma) \times H^{-\half}(\Gamma)$\\  \hline
\end{tabular}}
\caption{Overview of problems (P$_\boldbeta$) with their respective BCs, domain and function spaces. Remark that for $\beta=3$, $\mH(\Gamma) = \IH(\Gamma)$, and that notations allow to define $\boldbeta$-independent function spaces for the BVPs.}
\label{tab:Notations_Problem}
\end{table}
\begin{subproblem}[P$_\beta$) ($\beta=0,1,2$]Given $\kappa>0$ and $\U^\textup{inc} \in H_{\textup{loc}}^1(D^c)$ with $L_\kappa \U^\textup{inc}=0$ in $D^c$, we seek $\U\in H_{\textup{loc}}^1(D^c)$ such that
$$
\left\{
\begin{array}{lll}
\Delta\U+ \kappa^2\U  = 0&\quad\text{in}\quad D^c,\\
\noalign{\vspace{3pt}}
  \gamma_{\beta} \U = 0&\quad\text{on}\quad \Gamma&\quad \text{if}\quad \beta \in \{0,1\}\quad\text{or}\\
  \noalign{\vspace{3pt}}
  \gamma_{1} \U  + \imath\eta \gamma_0 \U = 0,~\eta >0 &\quad\text{on}\quad\Gamma & \quad\text{if}\quad\beta = 2,\\
  \noalign{\vspace{3pt}}
    \SRC(\U^\textup{sc},\kappa).
\end{array}
\right.
$$
\end{subproblem}
\begin{subproblem}[P$_\textup{3}$]Let $\kappa_i,\mu_i>0$, $i=0,1$, with either $\kappa_0\neq \kappa_1$ or $\mu_0\neq\mu_1$, and $\U^\textup{inc} \in H_{\textup{loc}}^1(D^c)$ with $L_{\kappa_0} \U^\textup{inc}=0$ in $D^c$. We seek $(\U^0,\U^1) \in  H_{\textup{loc}}^1(D^c) \times H^1(D)$ such that
$$
\left\{
\begin{array}{lll}
\Delta \U^i+\kappa^2_i\U^i &= 0&\quad\text{in}\quad D^i\quad\text{for}\quad i=0,1,\\
\noalign{\vspace{3pt}}
  [\gamma_{0} \U]_\Gamma &= 0 &\quad\text{on}\quad \Gamma,\\
  \noalign{\vspace{3pt}}
  [\mu^{-1}\gamma_{1} \U]_\Gamma &= 0&\quad\text{on}\quad \Gamma,\\
  \noalign{\vspace{3pt}}
\SRC(\U^\textup{sc},\kappa_0).
\end{array}
\right.
$$
\end{subproblem}

\section{Main results}\label{sec:main_results}
We directly present the main results of this work, namely a precise characterization of the differentiability properties of the domain-to-solution and domain-to-Cauchy data mappings. For the sake of simplicity, we implemented an open-source Jupyter Notebook GUI\footnote{github/pescap/shift} that sums up those results in a easy-to-use manner. It allows to define the case of interest and obtain the respective regularity output spaces for all quantities concerned.
\begin{theorem}[MD]\label{thm:MD}Set $\fr \in \IN_1$ and $\fq \geq 1$. Consider a domain $\Gamma \in C^{\fr,1}$, the solution of (P$_\boldbeta$), and set
\be 
\dot{\fs} : =  \left\{ \begin{array}{lll}\overline{\fs}:=\min(\fq,\fk) \quad \text{for}\quad\beta = 0,\\\noalign{\vspace{2pt}}
\underline{\fs}:=\min\left(\fq, \fr - \half, \fk\right)\quad  \text{ for}\quad\beta = 1,2,3.\end{array}\right.
\ee
Then, the maps
\be
\begin{array}{llll}
\mfT  &\to & \fH^{\dot{\fs}+1}(\fD) & : \btheta \mapsto \U_\theta \circ T_\theta,  \\\noalign{\vspace{2pt}}
\mfT  &\to &\IH^{\dot{\fs}}(\Gamma) & : \btheta \mapsto (\gamma \U_\theta) \circ T_\theta,
\end{array}
\ee
\begin{enumerate}
  \item[(i)] are Fr\'echet differentiable in a neighborhood from zero for $\mfT \subseteq W^{\fr+1, \infty}(\fD)$;
  \item[(ii)] are $C^1$  for $\mfT \subseteq W^{\fk+1,\infty}(\fD)$;
  \item[(iii)] have the following stability estimates:
\begin{align*}
\|\U_\theta \circ T_\theta - \U\|_{\fH^{\dot{\fs}+1}(\fD)} &\leq C \|\btheta\|_{W^{\fk+1,\infty}(\fD)},\\
\|(\gamma\U_\theta) \circ \theta - \gamma\U\|_{\IH^{\dot{\fs}}(\Gamma)} &\leq C  \|\btheta\|_{W^{\fk+1,\infty}(\fD)}.
\end{align*}
\end{enumerate}
\end{theorem}

\begin{theorem}[SD]\label{thm:SD}Set $\fr \in \IN_1$ and $\fq \geq 0$. Consider a domain $\Gamma \in C^{\fr,1}$, the solution of (P$_\boldbeta$), and set
\be 
\tilde{\fs} : = \min\left( \fq-1, \fr - \half, \fk + \half\right).
\ee
Then, it holds that the maps 
\be
\begin{array}{llll}
\mfT &\to &\fH^{\tilde{\fs} + 1 } (\fD) & : \btheta \mapsto \U_\theta,  \\\noalign{\vspace{2pt}}
\mfT &\to &\IH^{\tilde{\fs}}(\Gamma)& : \btheta \mapsto \gamma \U_\theta,
\end{array}
\ee
\begin{enumerate}
  \item[(i)] are Fr\'echet differentiable in a neighborhood from zero for $\mfT \subseteq W^{\fr+ 1 , \infty}(\fD)$;
  \item[(ii)] are $C^1$  for $\mfT \subseteq W^{\fk+1,\infty}(\fD)$;
  \item[(iii)] have the following stability estimates:
\begin{align*}
\| (\tilde{\U}_\theta)  - \U\|_{\fH^{\tilde{\fs} + 1 } (\fD)} &\leq C  \|\btheta\|_{W^{\fk+1,\infty}(\fD)},\\
\| (\gamma\tilde{\U}_\theta)  - \gamma\U\|_{\IH^{\tilde{\fs}}(\Gamma)} &\leq C \|\btheta\|_{W^{\fk+1,\infty}(\fD)},
\end{align*}
\end{enumerate}
with tilde notation referring to a suitably defined extension of $\U_\theta$.
\end{theorem}
The formal proofs of \cref{thm:MD} and \cref{thm:SD} are achieved following the explicit shape calculus: we perform (S1) in \cref{sec:regularity} and finish the proof with (S2) in \cref{sec:Banach}.
\section{Shape calculus for Helmholtz scattering problems}\label{sec:Shape_calculus}
Before proving our main results, we introduce the tools required to perform the explicit shape calculus.
\subsection{Directional domain transformations}
\label{subs:domain_transformation}
Set $r \in \IN_0$ and consider $\Gamma \in C^{r,1}$. One can introduce the family of transformed boundaries:
\be
\label{eq:transformed_boundary}
\Gamma_t := \{ \bx + t \bv : \bx \in \Gamma \} = T_t (\Gamma),
\ee
with a \emph{velocity field} $\bv \in W^{k+1,\infty}(\Gamma;\IR^d)\equiv W^{k+1,\infty}(\Gamma)$, $0 \leq  k \leq r$. Clearly, $\Gamma_0\equiv\Gamma$. Following \cite[Section 2.8]{Sokolowski}, there is a small $\eps>0$ such that $T(t,\bx):=T_t(\bx)$ belongs to $T \in C^1\left((-\eps,\eps); W^{k+1,\infty}(\Gamma)\right)$, and that $T_t$ is an one-to-one bijective transformation from $\Gamma$ to $\Gamma_t:=T_t(\Gamma)$ for $|t| < \eps$. Fixing $|t|<\eps$, the perturbed boundary $\Gamma_t$ originates a domain $D_t$ of class $C^{k,1}$. Moreover, by the continuity of $\bv$ on the compact surface $\Gamma$, the velocity field can be extended to $D$, and even to $\IR^d$ by using a bounded hold-all domain such that $\bv=0$ outside of it---e.g., $B_R$ or even a $\Gamma$-tubular neighborhood $\Gamma_\mE$, with $\tilde{\Gamma} :=\fD \cap \Gamma_\mE$. Finally, we remark that $(\bv \cdot \bn) \in W^{\min(r,k+1),\infty}(\Gamma)$ and introduce the Jacobian matrix of $\bv$ denoted $\nabla \bv \in W^{k,\infty}(\Gamma)$ along with 
\be 
A'(0):= \divg \bv \bI - \nabla \bv - \nabla \bv^T \in W^{k,\infty}(\tilde{\Gamma}).
\ee
\subsection{Directional domain derivatives}
\label{subs:domain_derivatives}
With the domain transformation and velocity field defined in \cref{subs:domain_transformation}, we are ready to introduce both DDs.
\begin{definition}[Domain derivatives]
\label{def:domain_derivative}Consider a shape dependent scalar field $\U_t$ defined in $\fD_t$ corresponding to $D_t^c$ or $D_t^c \cup D_t$ for $|t|< \eps$ and denote $\U\equiv\U_0$ for the nominal domain solution. $\U_t$ admits a SD (resp.~a MD), denoted by  $\U'$ (resp.~$\dot{\U}$) in $\fD$ along $\bv$, if the following limit exists
\be
\label{eq:ddev}
\U':= \lim_{t \to 0}\frac{\U_t-\U}{t} \quad\text{or}\quad\dot{\U}:=\lim_{t \to 0}\frac{\U_t\circ T_t-\U}{t}. 
\ee
Assuming that $\dot{\U} \in \fH^q(\fD)$, $q \geq 0$, then the following Taylor expansion holds for $|t|<\eps$:
\be
\U_t \circ T_t(\bx) = \U(\bx) + t \dot{\U}(\bx)+o(t)\quad\text{in}\quad\fH^q(\fD).\nonumber
\ee
Assuming that $\U' \in \fH^q(\fD)$, $q \geq 0$, then the following local Taylor expansions hold for $|t|<\eps$ and any suitably defined extension $\tilde{\U}_t \in \fH^q(\fD) $:
\be \label{eq:SD_Taylor}
\begin{array}{ccl}
\tilde{\U}_t(\bx)&= \U(\bx) + t \U'(\bx)+o(t) &\text{in}\quad\fH^q(\fD)\quad\text{and}\\
\U_t(\bx)&= \U(\bx) + t \U'(\bx)+o(t) &\text{in}\quad \fH_\textup{kloc}^q(\fD).
\end{array}
\ee
Morever, the SD and MD are linked as follows
\begin{align}
\dot{\U}(\bx)&= \U'(\bx) + \nabla \U(\bx) \cdot \bv(\bx).\label{eq:lie}
\end{align}
\end{definition}

We first point out that \cref{eq:ddev} is the usual definition for the DDs. Secondly, \cref{eq:lie} relates the SD and MD modulo $\nabla \U \cdot \bv$---the Lie derivative (see \cite[Definition 1]{hiptmair2013shape}). Second notation in \cref{eq:SD_Taylor} uses the $\text{kloc}$-subscript representing the classic local integration space i.e.~in $\fH^q(\fK)$ on any compact $\fK \Subset \fD$. Extension of these DDs to Sobolev spaces will rely on their characterization as BVPs as detailed later on. Notice that due to its definition on varying domains, the shape Taylor expansion for the SD is only valid by using a smooth extension, or on compact subsets of $\fD$, in order to guarantee a conventional Banach setting. The latter is not amenable to define such functions close to the boundary of the domains. Similarly, we define the DDs of functions defined on the boundary $\Gamma$.
\begin{definition}[Boundary domain derivatives]
\label{def:boundary_domain_derivative}
Consider a shape dependent scalar field $u_t$ defined on $\Gamma_t$  for $|t|<\eps$. Thus, $u_t$ admits a SD (resp.~a MD), denoted $u'$ (resp.~$\dot{u}$) on $\Gamma$ in the direction $\bv$, if the following limits exist
\be
\label{eq:boundary_ddev}
u':= \lim_{t \to 0}\frac{u_t-u}{t} \quad\text{and}\quad\dot{u}:=\lim_{t \to 0}\frac{u_t\circ T_t-u}{t}, 
\ee
respectively. Assuming that $\dot{u} \in H^q(\Gamma)$, $q\in \IR$, the following Taylor expansion holds for $|t| <\eps$:
\be\label{eq:boundary_md}
u_t \circ T_t(\bx) = u(\bx) + t \dot{u}(\bx)+o(t)\quad\text{ in }\quad H^q(\Gamma).
\ee
Assuming that $u' \in H^q(\Gamma)$, $q\in \IR$ and that $u_t$ admits an extension $\tilde{u}_t\in H^q(\Gamma)$:
\be
\tilde{u}_t(\bx) = u(\bx) + t u'(\bx)+o(t)\quad\text{in}\quad H^q(\Gamma).\label{eq:boundary_sd}
\ee
Moreover, both DDs are related according to
\be
\dot{u}(\bx)= u'(\bx) + \nabla_\Gamma u(\bx) \cdot \bv(\bx).\label{eq:boundary_lie}
\ee
\end{definition}

\begin{remark}[Starshaped domains]
\cref{def:domain_derivative,def:boundary_domain_derivative} consider Taylor expansion in their least restrictive and most natural spaces. They clear the way for the introduction of adapted Banach spaces to derive differentiability properties. Besides, if $D$ is starshaped, we introduce the parametrization $\rho$ in \cref{subs:Lipschitz} and assume that $\bv(\rho)$. This allows for the restriction of shape calculus to the boundary as $\Gamma \equiv \rho(\phi)$, and to consider Taylor expansions in $I$, e.g.~in periodic functions on $C_\text{per}^k(I)$ (refer to \cite[Remark 3]{eppleroptimal}).\end{remark}

\subsection{Material derivatives for Helmholtz scattering problems}
\label{subs:helmholtz_scattering_material}
Consider a Lipschitz domain and transformation. From Kleemann \cite{kleemann2012shape}, we summarize the problems (MP$_\boldbeta$) satisfied by the associated MD.
\begin{subproblem}[MP$_\beta$) ($\beta=0,1,2$]\label{pb:MD_P012}We seek $\dot{\U}$ as the solution of
$$
\left\{
\label{eq:M_EP}
\begin{array}{lll}
\Delta\dot{\U}+\kappa^2\dot{\U} = f_\beta &\text{in}\quad D^c,\\
\noalign{\vspace{3pt}}
  \gamma_{\beta} \dot{\U} = m_\beta&\text{on}\quad\Gamma&\text{if}\quad\beta \in \{0,1\}\quad\text{or} \\
  \noalign{\vspace{3pt}}
  \gamma_{1} \dot{\U}  + \imath \eta \gamma_0 \dot{\U}=  m_2,~\eta >0 &\text{on}\quad\Gamma&\text{if}\quad \beta = 2,\\
  \noalign{\vspace{3pt}}
  \SRC(\dot{\U},\kappa),
\end{array}
\right.
$$
wherein, for $\U$ being the respective solution of \textup{(P}$_\beta$\textup{)}, we have
\begin{align*}
f_\beta &=
 \divg (A'(0) \nabla \U) + \kappa^2 \divg(\bv) \U,\\
m_0 &: =  0,\\
m_1 & : = - \nabla \bv^T \cdot \U  \cdot \bn  + \nabla \U \cdot \nabla_\Gamma (\bv \cdot \bn),\\
m_2 &: = - \nabla \bv^T \cdot \U  \cdot \bn  + \nabla \U \cdot \nabla_\Gamma (\bv \cdot \bn) + \divg_\Gamma \bv \gamma_1 \U.
\end{align*}
\end{subproblem}
\begin{subproblem}[MP$_3$]\label{pb:MD_P3}We seek $\dot{\U}=(\dot{\U}^0,\dot{\U}^1)$ as the solution of
$$
\left\{
\label{eq:M_EP_3}
\begin{array}{lll}
\Delta \dot{\U}^i + \kappa^2_i \dot{\U}^i &= f^i_3&\text{in}\quad D^i\text{, for } i=0,1,\\
\noalign{\vspace{3pt}}
  [ \gamma_{0}  \dot{\U} ]_\Gamma&= m^0_3 &\text{on}\quad \Gamma,\\
  \noalign{\vspace{3pt}}
  [ \frac{1}\mu \gamma_{1} \dot{\U} ]_\Gamma &= m^1_3 &\text{on}\quad \Gamma,\\
  \noalign{\vspace{3pt}}
  \SRC(\dot{\U}^0,\kappa_0),
\end{array}
\right.
$$
with boundary data built using $\U$ as the solution of \textup{(P$_3$)}, as follows
\begin{align*}
f_3^i &:= 
\divg (A'(0) \nabla \U^i) + \kappa^2 \divg(\bv) \U^i,\quad i=0,1,\\
m_3^0 & :=  0, \\
m_3^1 & :=- \Big[ \frac{1}{\mu} \Big]_\Gamma \nabla \bv^T \cdot \U|_\Gamma  \cdot \bn  + [\frac{1}{\mu}\nabla \U]_\Gamma \cdot \nabla_\Gamma (\bv \cdot \bn).
\end{align*}
\end{subproblem}
\subsection{Shape derivatives for Helmholtz scattering problems}
\label{subs:helmholtz_scattering_shape}
Consider a Lipschitz domain and transformation. The problems (SP$_\boldbeta$) satisfied by the corresponding SD \cite{hiptmair2017shape} are
\begin{subproblem}[SP$_\beta$) ($\beta=0,1,2$]We seek $\U'$ as the solution of
$$
\left\{
\label{eq:S_EP}
\begin{array}{lll}
\Delta\U'+\kappa^2\U' = 0&\text{ in } D^c,\\
\noalign{\vspace{3pt}}
  \gamma_{\beta} \U' = g_\beta&\text{on}\quad \Gamma&\text{if}\quad\beta \in \{0,1\}\quad \text{or} \\
  \noalign{\vspace{3pt}}
  \gamma_{1} \U'  + \imath \eta \gamma_0 \U'=  g_2,~\eta >0 &\text{on}\quad \Gamma &\text{if}\quad \beta = 2,\\
  \noalign{\vspace{3pt}}
  \SRC(\U',\kappa),
\end{array}
\right.
$$
wherein, for $\U$ being the respective solution of \textup{(P$_\beta$)}, we have
\begin{align*}
g_0 &: =  - \gamma_1\U (\bv \cdot \bn),\\
g_1 &: = \divg_{\Gamma} \big((\bv \cdot \bn) \nabla_{\Gamma} \U\big) + \kappa^2 \gamma_0\U ( \bv \cdot \bn ),\\
g_2 & : =  \divg_{\Gamma} \big((\bv \cdot \bn) \nabla_{\Gamma} \U\big) + \kappa^2 \gamma_0\U ( \bv \cdot \bn ) + \imath \eta  (\bv \cdot \bn) (-\gamma_1 \U- \mathfrak{H} \gamma_0 \U).
\end{align*}
\end{subproblem}
\begin{subproblem}[SP$_3$]We seek $\U'=(\U'^0,\U'^1)$ as the solution of
$$
\left\{
\label{eq:S_EP_3}
\begin{array}{lll}
\Delta \U'^i + \kappa^2_i \U'^i &= 0&\text{in}\quad D^i\quad \text{for}\quad i=0,1,\\
\noalign{\vspace{3pt}}
  [ \gamma_{0}  \U' ]_\Gamma&= g_3^0 &\text{on}\quad\Gamma,\\
  \noalign{\vspace{3pt}}
  [ \frac{1}\mu \gamma_{1}  \U' ]_\Gamma &= g_3^1 &\text{ on}\quad \Gamma,\\
  \noalign{\vspace{3pt}}
  \SRC(\U'^0,\kappa_0),
\end{array}
\right.
$$
with boundary data built using $\U$ as the solution of \textup{(P$_3$)}, as follows
\begin{align*}
g^0_3 & :=  -  [ \gamma_1\U ]_\Gamma (\bv \cdot \bn), \\
g_3^1 & :=  \Big[ \frac{1}{\mu} \Big]_\Gamma\divg_{\Gamma} \big((\bv \cdot \bn) \nabla_{\Gamma} \U\big) +  [ \kappa^2  ]_\Gamma\gamma_0\U ( \bv \cdot \bn ).
\end{align*}
\end{subproblem}
\begin{remark}
The BVPs for the SD have a relatively simple form: by virtue of the Hadamard structure theorem, they only involve the normal component of the velocity field on the boundary, while the MD presents a source term and a more complex form. The characterizations are valid in Lipschitz domains \cite{djellouli1999characterization}.
\end{remark}

\subsection{Shape calculus for the Cauchy data}
\label{sec:shdev_cauchy}
Hereafter, we provide a description of the CDDs arising in (P$_\boldbeta$).
\begin{theorem}[Shape Cauchy data]
\label{thm:shape_cauchy}
Let $\U$, $\dot{\U}$ and $\U'$ denote the solutions of \textup{(P$_\boldbeta$)}, \textup{(MP$_\boldbeta$)} and \textup{(SP$_\boldbeta$)}, respectively. The CDDs are given by
\be\label{eq:CMD}
\dot{\bxi} := \begin{pmatrix}
\gamma_0\dot{\U}\\
\gamma_1\dot{\U} - \nabla \bv^T  \U \cdot \bn - \nabla_\Gamma \U \cdot \nabla_\Gamma (\bv \cdot \bn)
\end{pmatrix},
\ee
and 
\be\label{eq:CSD}
\bxi' := \begin{pmatrix}
\gamma_0\U'  + \gamma_1 \U (\bv \cdot \bn)\\
\gamma_1 \U' +  \gamma_2 \U (\bv \cdot \bn) -\nabla_\Gamma \U \cdot \nabla_\Gamma (\bv \cdot \bn)
\end{pmatrix}.
\ee
Besides, the following holds
\be
\label{eq:CLD}
\dot{\bxi} = \bxi' + \nabla_\Gamma \bxi \cdot\bv.
\ee
\end{theorem}

\begin{proof}Formula for the CMD are deduced using \cite[Equation 25]{bochniak2002domain} and by taking $\dot{\bn} = \bn' = - \nabla_\Gamma (\bv\cdot \bn)$ \cite[Equation 4.13]{hiptmair2017shape}), giving $\dot{\lambda} = \gamma_0\dot{\U}$, and:
\be 
\begin{array}{ll}
\dot{\sigma}&=  \nabla \dot \U \cdot \bn  - \nabla \bv^T  \U \cdot \bn + \nabla\U \cdot \dot{\bn}\\
& = \gamma_1\dot{\U} - \nabla \bv^T \U \cdot \bn - \nabla \U \cdot \nabla_\Gamma (\bv \cdot \bn).
\end{array}
\ee
For the CSD, we follow the classical steps of differentiation of shape functionals defined on a boundary \cite[Section 5.6]{antoine_var}, which we reproduce for completeness. For the case of Dirichlet traces, we consider the identity
\be
\lambda_t = \U_t|_{\Gamma_t}.
\ee
Upon differentiating this formula, we obtain
\be
\lambda' + \mathfrak{H}  \lambda (\bv \cdot \bn) =  \gamma_0 \U ' + \frac{\partial \U}{\partial \bn} (\bv \cdot \bn) +\mathfrak{H} \gamma_0 \U (\bv \cdot \bn),
\ee
and consequently, the desired result for first element of (\ref{eq:CSD}) since $\lambda = \gamma_0 \U$, and for (\ref{eq:CLD}) by using (\ref{eq:boundary_lie}).
For Neumann traces, as $\sigma_t = \nabla \U_t |_{\Gamma_t} \cdot \bn_t $, differentiation leads to
\be
\sigma' + \mathfrak{H} \sigma (\bv \cdot \bn)= \big(\nabla \U_t|_{\Gamma_t} \cdot \bn_t\big)' + \frac{\partial^2 \U }{\partial \bn^2} (\bv \cdot \bn) + \mathfrak{H} \frac{\partial \U}{\partial \bn} (\bv \cdot \bn),
\ee
Hence, by the product rule and since $\sigma = \gamma_1\U$, we derive
\be
\sigma' = \gamma_1 \U' - \nabla_\Gamma \U \cdot \nabla_\Gamma (\bv \cdot \bn)+ \frac{\partial^2 \U }{\partial \bn^2} (\bv \cdot \bn),
\ee
yielding second element in (\ref{eq:CSD}). 
\end{proof}
The above formulas show that the CSD depend on the solution of the shape BVPs as well as on $(\bv \cdot \bn)$, $\nabla_\Gamma (\bv \cdot \bn),\nabla_\Gamma \U$ and $\gamma_\beta\U$, for $\beta \in \{0,1,2\}$. Therefore, they can be easily evaluated when using numerical methods based on BIEs, which provide Cauchy traces, and whose second-order traces can be accessed through vector calculus.

\section{Characterization Step (S1)}
\label{sec:regularity}
In this section, we focus on the characterization step (S1), and follow a three-step approach:
\begin{enumerate}
  \item[(i)] provide regularity results for $\U$ via elliptic regularity estimates for (P$_\boldbeta$);
  \item[(ii)] deduce regularity for the DDs' data with trace embeddings;
  \item[(iii)] determine regularity estimates for the MD (resp.~SD) through elliptic regularity estimates for (MP$_\boldbeta$) (resp.~(SP$_\boldbeta$)).
\end{enumerate}
We start by tackling smooth domains using general shift theorems in \cref{subsec:Classic_Shift}. Then, we consider limit shift theorems, and domains with corners. For the upcoming analysis, we recall a result that will be employed repeatedly in \cref{sec:regularity}.
\begin{lemma}[Multipliers]\label{lemma:product}Consider a $C^{r,1}$-domain for some $r\in\IN_0$. For $\sigma \in H^{q}(\Gamma)$ with $q\in\IR$, $|q| \leq r+1$, and $\phi \in W^{|q|,\infty}(\Gamma)$, it holds that
\be
\|\sigma \phi \|_{H^q(\Gamma)} \leq \|\sigma\|_{H^q(\Gamma)} \|\phi\|_{W^{|q|,\infty}(\Gamma)}.
\ee
\end{lemma}
\begin{proof}
This result is proved in \cite[section 6]{sparse3} with respect to $C^{|q|-1,1}(\Gamma)$ functions instead of $W^{|q|,\infty}(\Gamma)$, leading to the same result directly as $\|\phi\|_{C^{|q|+1,1}(\Gamma)} \leq C \|\phi\|_{W^{|q|,\infty}(\Gamma)}$ by Morrey's inequality (see \cite[section 5.6.2]{evans2010partial}). 
\end{proof}
\subsection{Classic shift regularity theorem for regular domains}\label{subsec:Classic_Shift}
The choice of a relatively smooth domain of index $\fr$ and data $\fk$ induces smoothness for the BVPs as stated in \cref{lemma:shift} (see e.g.,~\cite{MacLean}). 
\begin{lemma}[Classic shift theorem]
\label{lemma:shift}
Consider a domain of class $C^{\fr,1}$, $\fr\in \IN_0$, assume that $\xi^\textup{inc}\in \mH^\fq(\Gamma)$, $\fq \leq \fr$ and define $\U$ to be the unique solution of \textup{(P$_\boldbeta$)}. Then,
\be
\U \in \fH^{\fq+1}(\fD) \textup{, and  } \gamma  \U \in \IH^{\fq}(\Gamma)
\ee
and
\be
\|\U \|_{\fH^{\fq+1}(\fD)}\leq C \|\xi^\textup{inc} \|_{\mH^{\fq}(\Gamma)}.
\ee
\end{lemma}
\begin{proof}
By the absence of source terms---and the smoothness of the right-hand sides, we apply classical regularity results for strongly elliptic operators defined in domains of class $C^{\fr,1}$ to the scattered field. Thus, we conclude that $\U^\textup{sc} \in H_\loc^{\fq+1}(\fD)$ along with continuous dependence on the data. These results are detailed in \cite[Theorem 2.5.1.1]{Grisvard} for (P$_\beta$), $\beta = 0,1,2$, for bounded domains and are easily adapted to weighted spaces $\fH_\kappa^{\fq+1}(\fD)$ (see also \cite[Theorem 2.5.21]{Nedelec} for mapping properties in exterior domains or \cite[Theorem 4.20]{MacLean} for the transmission problem). Finally, the relation between the scattered and total field and continuity of traces \cite[Theorem 1.5.1.5]{Grisvard} provide the desired result.
\end{proof}

Now, we can characterize accurately the behavior of solutions for all the BVPs considered so far with respect to regularity indices $(\fr,\fq,\fk)$. We introduce the $\Rrightarrow$ symbol to alleviate notations and simplify regularity considerations for Sobolev spaces, making functional space and domain definition implicit. For instance, $\U \Rrightarrow \fr + 1$ stands for $\U \in \fH^{\fr+1}(\fD)$, or $\bv \Rrightarrow \fk+ 1 $ for $\bv \in W^{\fk+1,\infty}(\Gamma)$. In this context, $a\cap b$ refers to $\min(a,b)$.
\begin{lemma}\label{lemma:regularity_others}Consider the setting of \cref{lemma:shift} along with a transformation $\bv \in W^{\fk+1,\infty}(\Gamma)$, $0 \leq \fk \leq \fr$, and set $\fs_t := \min(\fq, \fk)$. Then, it holds that
\begin{align*}
\U_t \circ T_t &\in \fH^{\fs_t+ 1} (\fD_t), \quad 
\nabla \U \cdot \bv \in \fH^{\min(\fq - 1 ,\fk) + 1} (\fD)\quad \text{and}\\
\nabla_\Gamma \U \cdot \bv_\Gamma & \in H^{\half + \min\left(\fq-1,\fk+\half\right)}(\Gamma).
\end{align*}
\end{lemma}
\begin{proof}
Arguments of \cref{lemma:shift} applied to $\U_t$ in $\fD_t \in C^{\fk,1}$ lead to $\U_t  \Rrightarrow \min(\fq , \fk) + 1= \fs_t + 1$, thence $\U_t\circ T_t \Rrightarrow \fs_t +1$. Next, \cref{lemma:shift} states that $\U\Rrightarrow 1 + \fq$, giving $\nabla \U \Rrightarrow \fq$ and $\nabla \U \cdot \bv \Rrightarrow \min(\fq, \fk+1)$. Besides, $\nabla_\Gamma \U \Rrightarrow \fq -\half$ giving $\nabla_\Gamma \U\cdot \bv \Rrightarrow \min\left( \fq-\half, \fk + 1\right) = \half + \min \left(\fq -1, \fk + \half\right)$. 
\end{proof}
We are ready to state the regularity results for the DDs, starting with the MD.
\begin{theorem}
\label{thm:regularityMD}
Consider the setting of \cref{lemma:shift} along with $\bv \in W^{\fk+1,\infty}(\Gamma)$ and
\begin{align*}
\fr \in \IN_1,\quad 1\leq \fq \leq \fr ,\quad 1 \leq \fk\leq \fr.
\end{align*}
Set 
\be 
\dot{\fs} : =  \left\{ \begin{array}{lll}\overline{\fs}:=\min(\fq, \fk) \quad \quad\text{for}\quad \beta = 0,\\\noalign{\vspace{2pt}}
\underline{\fs}:=\min\left(\fq,\fr-\half,\fk\right) \quad \text{else}.\end{array}\right.
\ee
Then
\be 
\dot{\U} \in \fH^{\dot{\fs} + 1} (\fD)\textup{, and } \dot{(\gamma \U)} \in \IH^{\dot{\fs}}(\Gamma).
\ee
\end{theorem}
\begin{proof}We start by characterizing the source term in \cref{pb:MD_P012}:
\begin{align*}
f_\beta & = \divg (A'(0) \nabla \U) + \kappa^2 \divg(\bv) \U\\
&\Rrightarrow (\min(\fk, \fq) -1)\cap \min(\fk, \fq+1)  =-1 + \min(\fq,\fk)= -1 + \overline{\fs}.
\end{align*}
Identically, we obtain the same regularity estimate for $\beta = 3$. Therefore
\be
\|f_\boldbeta\|_{\fH^{\overline{\fs} -1 } (\fD)}\leq C \|\U\|_{\fH^{\fq + 1} (\fD)}  \|\bv\|_{W^{\fk+1,\infty}(\fD)}.
\ee
Next, we set $\dot{\fs}_{m} : = \min\left(\fq, \fr-\half, \fk+\half\right)$. The data read $m_0=0$, and by \cref{lemma:product}, it holds that
\begin{align*}
m_1 & : = - \nabla \bv^T \cdot \U  \cdot \bn  + \nabla_\Gamma \U \cdot \nabla_\Gamma (\bv \cdot \bn)\\
& \Rrightarrow \min\left( \fk ,\half+ \fq, \fr\right) \cap \min\left(-\half+\fq, \fk ,\fr-1\right) \\
& \Rrightarrow \min\left(-\half + \fq , \fr-1, \fk\right) = -\half + \dot{\fs}_m\textup{, and}\\
m_2 &: = - \nabla \bv^T \cdot \U  \cdot \bn  + \nabla_\Gamma \U \cdot \nabla_\Gamma (\bv \cdot \bn) + \divg_\Gamma \bv \gamma_1 \U\\
& \Rrightarrow \min\left(\fk, \half + \fq, \fr \right) \cap \min\left( -\half+\fq , \fk, \fr-1 \right) \cap \min\left(\fk , -\half + \fq\right)\\
& \Rrightarrow \min\left(-\half+\fq, \fr-1, \fk\right)= -\half + \dot{\fs}_m.
\end{align*}
Besides, for $\beta = 1,2,3$
\be 
\label{eq:cd_mbeta}
\|m_\beta\|_{\mH^{\dot{\fs}_m}(\Gamma)}\leq C \|\xi_\beta\|_{\mH^\fq(\Gamma)}\|\bv\|_{W^{\fk+1,\infty}(\Gamma)}.
\ee
To summarize, $f_\boldbeta\Rrightarrow -1 + \overline{\fs}$, and $m_1$, $m_2$, $m_3\Rrightarrow -\half + \dot{\fs}_m$. Provided that $\overline{\fs}, \dot{\fs}_m >0$, we deduce by elliptic regularity that $\dot{\U} \in \fH^{\dot{\fs} + 1} (\fD)$, with $\dot{\fs}: = \overline{\fs}$ for $\beta=0$, and $\dot{\fs}:=\underline{\fs}$ for $\beta=1,2,3$. Next, according to \cref{eq:CMD}, $\dot{\lambda} = \dot{(\gamma_0\U)} = \gamma \dot{\U} \in H^{\dot{\fs}}(\Gamma)$. For $\beta=0$, since $\U|_\Gamma = 0$ hence $\nabla_\Gamma \U =0$, one can write $\dot{\sigma} = \gamma_1\dot{\U}\in H^{\overline{\fs}}(\Gamma)$ by \cref{eq:CSD}. Next, for $\beta=1,2,3$,
\begin{align*}
\dot{\sigma}& \Rrightarrow  \min\left(\fk-\half, \fq-\half\right) \cap \min\left(\fk, \fq +\half , \fr\right) \cap \min\left(-\half +\fq, \fk ,\fr-1\right)\\
& \Rrightarrow \min \left(-\half + \fq, \fr -1 , -\half + \fk\right)\\
& \Rrightarrow -\half + \min \left(\fq, \fr - \half, \fk\right)= -\half + \underline{\fs},
\end{align*}
thus the final result.
\end{proof}
\begin{corollary}\label{coro:MD}
Consider the setting of \cref{thm:regularityMD}. Then
\be\label{eq:MD_cont}
\|\dot{\U} \|_{\fH^{\dot{\fs}+1}(\fD)} \leq  C \big(\|f_\boldbeta \|_{\fH^{\dot{\fs}-1}(\fD)} + \|m_\boldbeta\|_{\mH^{\dot{\fs}}(\Gamma)} \big) \leq C\|\bv\|_{W^{\fk+1,\infty}(\fD)}
\ee
and 
\be 
\label{eq:CMD_cont}
\|\dot{(\gamma\U)} \|_{\IH^{{\dot{\fs}}}(\Gamma)} \leq C \|\bv\|_{W^{\fk+1,\infty}(\fD)}.
\ee
\end{corollary}
\begin{proof}
The proof of \cref{eq:MD_cont} is straightforward by continuous dependence on the data of $\dot{\U} \in \fH^{\fs+1}(\fD)$ in \cref{thm:regularityMD}. Next, $\dot{(\gamma\U)}\in \IH^{\dot{\fs}}(\Gamma)$ and \cref{eq:MD_cont} give \cref{eq:CMD_cont}.
\end{proof}
\begin{remark}[Lipschitz domains]
For Lipschitz domains, i.e.~for $\fr=0$, the previous analysis can be conducted similarly, but does not guarantee well-posedness for the MD for $\beta=1,2,3$ due to the $(\fr-\half)<0$ term in the boundary data. Still, for $\beta=0$, one retrieves an established well-posedness result---see e.g.,~\cite[Proposition 6.30]{Allaire} for Laplace. For example, for $\fq=\fk=0$, one gets $f_0 \in \fH^{-1}(\fD)$ hence $\U \in \fH^1(\fD)$.
\end{remark}
\begin{remark}[Source terms and tilde-spaces]
In \Cref{thm:regularityMD}, we assumed restrictions $\fq\geq 1$ and $\fk \geq 1$ to ensure the source term $f_\boldbeta$ to be at least in $L^2$. The case when $\fq,\fk \in [0,1]$---for non-smooth data---raises two technical issues. To begin with, problems $\beta=1,2,3$ require $f_\beta$ to be in the tilde Sobolev space $\widetilde{\fH}^{-t}(\fD)$ \cite{MacLean} for $0\leq t \leq 1$, which is not guaranteed under this setting, as one obtains results in $\fH^{-t}({\fD})$. Furthermore, non-smooth data lead to shift results for $f_\beta \in \fH^{-\half}(\fD)$, which is a problematic and rather technical issue \cite[Section 1]{amrouche2011regularity}.
\end{remark}
Next, we derive a similar result for the SD.
\begin{theorem}
\label{thm:regularitySD}
Consider the setting of \cref{lemma:shift} along with a transformation $\bv \in W^{\fk+1,\infty}(\Gamma)$, $0 \leq \fk \leq \fr$ and set:
\be 
\tilde{\fs} : = \min\left( \fq-1,\fk + \half\right).
\ee
Then, it holds that
\be
\U' \in \fH^{\tilde{\fs}+1}(\fD)\quad \text{and} \quad(\gamma \U)' \in \IH^{\tilde{\fs}}(\Gamma).
\ee
\end{theorem}
\begin{proof}
As the SD has no source term (refer to \cref{subs:helmholtz_scattering_shape}), we can move straightly to the data:
\begin{align*}
g_0 =&  - \gamma_1\U (\bv \cdot \bn)\\
\Rrightarrow & \min \left(\fq -\half, \fk+1,\fr\right) =\half + \min\left(\fq-1, \fk + \half\right),\\\noalign{\vspace{3pt}}
g_1 =& \divg_{\Gamma} \big((\bv \cdot \bn) \nabla_{\Gamma} \U\big) + \kappa^2 \gamma_0\U ( \bv \cdot \bn )\\
\Rrightarrow & \left(\min (\fk+1,\fr,\fq-\half)-1\right) \cap \min\left(\fq+\half,\fk+1,\fr\right)\\
 \Rrightarrow& \min\left(\fk , \fr-1 ,\fq-\frac{3}{2}\right)= -\half + \min\left(\fq-1,\fk+\half\right),\\\noalign{\vspace{3pt}}
g_2 = & g_1 -\imath \eta  (\bv \cdot \bn)\gamma_1 \U  -\imath \eta  (\bv \cdot \bn) \mathfrak{H} \gamma_0 \U\\
\Rrightarrow & \left(-\half + \tilde{\fs}\right) \cap \min\left(\fk+1, \fr\right) \cap \min\left(\fk+1,\fr, \fr-1, \fq+\half\right)\\
\Rrightarrow & \left(-\half + \tilde{\fs}\right) \cap \min\left(\fk+1, \fr-1, \fq+\half\right) \Rightarrow -\half+\tilde{\fs}.
\end{align*}
Consequently, as $\tilde{\fs}>0$, we conclude that $\U'\in \fH^{\tilde{\fs}+1}(\fD)$. Besides, $\gamma\U' \in \IH^{\tilde{\fs}}(\Gamma)$. Next, concerning the CSD, it holds that
\begin{align*}
\lambda' & =\gamma_0\U'  + \gamma_1 \U (\bv \cdot \bn)\\
&\Rrightarrow \min \left(\half + \tilde{\fs}\right),\\\noalign{\vspace{2pt}}
\sigma' & = \gamma_1 \U' +  \gamma_2 \U (\bv \cdot \bn) -\nabla_\Gamma \U \cdot \nabla_\Gamma (\bv \cdot \bn)\\
& \Rrightarrow \left(-\half + \tilde{\fs}\right) \cap \min\left(-\frac{3}{2} + \fq , \fk + 1, \fr\right)\cap \min\left(-\half+ \fq , \fk, \fr-1\right)\\
& \Rrightarrow \left(-\half + \tilde{\fs}\right),
\end{align*}
hence the result for the CSD.
\end{proof}
\begin{corollary}\label{coro:SD}
Consider the setting of \cref{thm:regularitySD}. Then, one has that
\be\label{eq:SD_cont}
\|\U' \|_{\fH^{\tilde{\fs}+1}(\fD)} \leq C \|g_\boldbeta\|_{\mH^{\tilde{\fs}}(\Gamma)} \leq C \|\bv\|_{W^{\fk+1,\infty}(\Gamma)}
\ee
and 
\be\label{eq:CSD_cont}
\|(\gamma \U)'\|_{\IH^{\tilde{\fs}} } \leq C\|\bv\|_{W^{\fk+1,\infty}(\Gamma)}.
\ee
\end{corollary}
\begin{proof}
The proof of \cref{eq:SD_cont} follows by continuous dependence on the data of $\U'\in \fH^{\tilde{\fs}+1}(\fD)$ in \cref{thm:regularitySD}. Next, $(\gamma\U)'\in \IH^{\tilde{\fs}}(\Gamma)$ and \cref{eq:SD_cont} give \cref{eq:CSD_cont}.
\end{proof}
\begin{remark}
The aforementioned BVPs for SDs allow very simple characterizations for the DDs yielding precise descriptions of the solutions' smoothness. We retrieve a known result: the SD loses one order of differentiability when compared to $\U$ for smooth deformation fields, e.g.,~for $\fr \geq \fk+\frac{3}{2}$, giving $\U \in {\fH}^{\fr+1}(\fD)$ and $\U' \in {\fH}^\fr(\fD)$. Still, we show that under less regular velocity fields, i.e.~for $\fk+\frac{3}{2} \leq \fr$, $\bv$ governs the smoothness of both he SD and the perturbed solution, leading to $\U_t$, $\dot{\U}\in {\fH}^{\fk+1}(\fD)$, whereas $\U' \in {\fH}^{\fk+\frac{3}{2}}(\fD)$. The latter seems to be a new result: in this configuration, we get a sharper regularity result for the SD than for the perturbed solution and the MD. 
\end{remark}
\subsection{Sharp shift theorem for regular domains}\label{subsec:Sharp_Shift}
So far, the results did not consider limit cases, i.e.~$\fq \geq \fr$. Besides, the scattering problem does not present source terms, allowing to use sharp regularity results. To do so, we will use the theory of Jerison and Kenig \cite{jerison1981dirichlet,jerison1981neumann} based on harmonical analysis. 
Now, we apply the steps in \cref{subsec:Classic_Shift} to the limit case. To start with, we present a sharp regularity result extending \cref{lemma:shift}. For the sake of generality, we state the result for general Lipschitz domains.
\begin{theorem}[Sharp shift theorem]
\label{thm:sharp_shift}
Consider a domain of class $C^{\fr,1}$, $\fr\in \IN_0$, assume that $\xi^\textup{inc} \in \mH^{\fq}(\Gamma)$, $\fq \geq 0$ and define $\U$ to be the unique solution of \textup{(P$_\boldbeta$)}. Then, for $\fs := \min\left(\fq , \fr+\half\right)$, it holds that
\be
\U \in \fH^{\fs +1}(\fD),\quad \gamma  \U \in \IH^{\fs}(\Gamma)
\ee
and
\be
\|\U \|_{\fH^{\fs+1}(\fD)}\leq C \|\xi^\textup{inc} \|_{\mH^{\fs}(\Gamma)}.
\ee
\end{theorem}

\begin{proof}
Remark that the case $\fq \leq \fr$ corresponds to the classic shift theorem. Now, consider the limit case $\fq= \fr+\half$. Since there is no source term, we apply Dalhberg estimates and deduce that $\U \in \fH^{\fr +\frac{3}{2}}(\fD)$ \cite{costabel1990remark}. In this case, we use the continuity of the Dirichlet-to-Neumann map \cite{MacLean} in the limit case to deduce that $\gamma \U \in \IH^\fs (\fD)$. The range $\fq < \fr + \half$ follows by interpolation, and for $\fq > \fr+\half$, smoothness is limited by $\fr$.
\end{proof}
Consequently, and following \cref{lemma:regularity_others}, we remark that $\U_t \in \fH^{\min \left(\fq, \fk +\half\right)+1}(\fD)$, and $\U_t\circ T_t \in \fH^{\min (\fq , \fk) +1 } (\fD)$. Hence, we have only sharpen estimate for the boundary term:
\be 
\nabla_\Gamma \U \cdot \bv_\Gamma \in H^{\half +\min \left(\fq-1,\fr-\half,\fk + \half\right)}(\Gamma).
\ee
This allows us to extend \cref{lemma:regularity_others} as follows
\begin{lemma}\label{lemma:sharp_regularity_others}Consider the setting of \cref{thm:sharp_shift} along with a transformation $\bv \in W^{\fk+1,\infty}(\Gamma)$, $0 \leq \fk \leq \fr$, and set $\fs_t := \min\left(\fq, \fk+\half\right)$. Then
\begin{align*}
\U_t \circ T_t &\in \fH^{\fs_t+ 1} (\fD_t), \quad 
\nabla \U \cdot \bv \in \fH^{\min(\fq - 1 ,\fk) + 1} (\fD)\quad \textup{and}\\
\nabla_\Gamma \U \cdot \bv_\Gamma & \in H^{\half + \min\left(\fq-1,\fk+\half\right)}(\Gamma).
\end{align*}
\end{lemma}
As a consequence, we present the generalization of \cref{thm:regularityMD}.
\begin{theorem}
\label{thm:sharp_regularityMD}
Consider the setting of \cref{lemma:shift} along with a transformation $\bv \in W^{\fk+1,\infty}(\Gamma)$, $0 \leq \fk \leq \fr$, and set:
\be 
\dot{\fs} : =  \left\{ \begin{array}{lll}\overline{\fs}:=\min(\fq,\fk) \quad \text{for}\quad\beta = 0,\\\underline{\fs}:=\min\left(\fq, \fr - \half, \fk\right)\quad  \text{ for}\quad\beta = 1,2,3.\end{array}\right.
\ee
Then
\be 
\dot{\U} \in \fH^{\dot{\fs} + 1} (\fD)\quad \text{and}\quad \dot{(\gamma \U)} \in \IH^{{\dot{\fs}}}(\Gamma).
\ee
\end{theorem}
\begin{proof}\label{proof:sharp_regularityMD}
We replicate the proof of \cref{thm:regularityMD} by replacing $\fq$ (classic) by $\min\left(\fq ,\fr+\half\right)$ (sharp). For $\beta = 0,1,2,$, we obtain 
\begin{align*}
f_{\beta} & = \divg (A'(0) \nabla \U) + \kappa^2 \divg(\bv) \U\\
& \Rrightarrow \left(\min \left(\fq, \fr + \half\right)-1 \right) \cap (\fk-1)\\
&\Rrightarrow -1 + \min \left(\fq , \fr + \half, \fk\right)\\
& \Rrightarrow -1 + \min(\fq, \fk).
\end{align*}
Identically, we obtain the same regularity estimate for $\beta=3$. Besides, we have that
\begin{align*}
m_1  &: = - \nabla \bv^T \cdot \U  \cdot \bn  + \nabla_\Gamma \U \cdot \nabla_\Gamma (\bv \cdot \bn)\\
&\Rrightarrow\min \left(\fk, \fq + \half, \fr + 1, \fr \right)\cap \left( \min \left(\fq + 1, \fr+ \frac{3}{2}\right)- \frac{3}{2}\right)\cap \min\left(\fk+1, \fr-1\right) \\
& \Rrightarrow  -\half + \min \left(\fq, \fr - \half,\fk + \half\right).
\end{align*}
Provided that $\divg_\Gamma \bv \gamma_1\U \Rrightarrow -\half + \min \left(\fq,\fr+\half, \fk+\half\right)$, we deduce that $m_2 \Rrightarrow -\half +\dot{\fs}$.
Using these regularity estimates, we finish the proof for the MD estimates and the CMD with the same arguments as in the proof of \cref{thm:regularityMD}. 
\end{proof}
\begin{corollary}\label{coro:sharp_MD}
Consider the setting of \cref{thm:regularityMD}. Then
\be\label{eq:sharp_MD_cont}
\|\dot{\U} \|_{\fH^{\dot{\fs}+1}(\fD)} \leq  C \big(\|f_\boldbeta \|_{\fH^{\dot{\fs}-1}(\fD)} + \|m_\boldbeta\|_{\mH^{\dot{\fs}}(\Gamma)} \big) \leq C \|\bv\|_{W^{k+1,\infty}(\fD)}
\ee
and 
\be 
\|\dot{(\gamma\U)} \|_{\IH^{{\dot{\fs}}}(\Gamma)} \leq C \|\bv\|_{W^{\fk+1,\infty}(\fD)}.
\ee
\end{corollary}
\begin{proof}We follow verbatim the proof of \cref{coro:MD}.
\end{proof}
The source term prevents these sharp estimates to provide really interesting new regularity results compared to the classic cases in \cref{thm:regularityMD}. Still, it allows to consider the rather more general configuration of \cref{thm:sharp_shift}. Next, the result concerning the SD is as follows:
\begin{theorem}
\label{thm:sharp_regularitySD}
Consider the setting of \cref{thm:sharp_shift} along with a transformation $\bv \in W^{\fk+1,\infty}(\Gamma)$, $0 \leq \fk \leq \fr$ and set:
\be 
\tilde{\fs} : = \min\left( \fq-1, \fr - \half, \fk + \half\right).
\ee
Then, it holds that
\be
\U' \in \fH^{\tilde{\fs}+1}(\fD)\textup{, and } (\gamma \U)' \in \IH^{\tilde{\fs}}(\Gamma).
\ee
\end{theorem}
\begin{proof}
We replicate the proof of \cref{thm:regularitySD} by replacing $\fq$ (classic) by $\min\left(\fq ,\fr+\half\right)$ (limit) as follows
\begin{align*}
g_0 : =&  - \gamma_1\U (\bv \cdot \bn)\\
& \Rrightarrow \min \left(\fq - \half, \fr, \fk + 1, \fr \right)\\
& \Rrightarrow \half + \min\left(\fq - 1, \fr - \half, \fk + \half\right) \\\noalign{\vspace{3pt}}
g_1 :=& \divg_{\Gamma} \big((\bv \cdot \bn) \nabla_{\Gamma} \U\big) + \kappa^2 \gamma_0\U ( \bv \cdot \bn )\\
&\Rrightarrow \min \left(\fk, \fr-1, \fq -\frac{3}{2}, \fr -1\right) \cap \left(\fr + \half, \fr + 1 , \fr, \fk + 1\right)\\
& \Rrightarrow -\half +\min\left( \fq - 1, \fr- \half,\fk  + \half\right)\\\noalign{\vspace{3pt}}
g_2 : = & g_1 -\imath \eta  (\bv \cdot \bn)\gamma_1 \U  -\imath \eta  (\bv \cdot \bn) \mathfrak{H} \gamma_0 \U\\
\Rrightarrow & \left(-\half + \tilde{\fs}\right) \cap \min\left(\fk+1, \fr\right) \cap \min\left(\fk+1,\fr, \fr-1, \fq+\half\right)\\
\Rrightarrow & \left(-\half + \tilde{\fs}\right) \cap \min\left(\fk+1, \fr-1, \fq+\half\right) \Rightarrow -\half+\tilde{\fs}.
\end{align*}
\end{proof}
\begin{corollary}\label{coro:sharp_SD}
Consider the setting of \cref{thm:sharp_regularitySD}. Then,
\be\label{eq:sharp_SD_cont}
\|\U' \|_{\fH^{\tilde{\fs}+1}(\fD)} \leq C \|g_\boldbeta\|_{\mH^{\tilde{\fs}}(\Gamma)} \leq {C} \|\bv\|_{W^{\fk+1,\infty}(\Gamma)},
\ee
and 
\be 
\|(\gamma \U)'\|_{\IH^{\tilde{\fs}} } \leq {C}\|\bv\|_{W^{\fk+1,\infty}(\Gamma)}.
\ee
\end{corollary}
\begin{proof}
We follow verbatim the proof of \cref{coro:SD}.
\end{proof}

\section{Banach space embedding and Fr\'echet differentiability}
\label{sec:Banach}
Following the spirit of the explicit shape calculus, we introduce tools to embed the directional derivative into suitable Banach spaces on fixed domains and to consider Fr\'echet differentiability. To this end, we extend the domain transformations (\ref{eq:transformed_boundary}):
\be
\Gamma_\theta := \{ \bx + \btheta : \bx \in \Gamma \} = (\bI + \btheta)(\Gamma) =  T(\Gamma),
\ee
with $\btheta \in W^{k+1,\infty}(\Gamma; \IR^d)$. Also, we introduce the open $\eps$-neighborhood for $\eps >0$:
\be
\mfT_{\eps}:= \{ \btheta : \|\btheta\|_{W^{k+1,\infty}(\Gamma;\IR^d)}  < \eps \}.
\ee
One can extend the spaces to $D$, and also to $\IR^d$ with $B_R(0)$. These definitions will allow to deduce Fr\'echet differentiability in $\btheta=0$. Next, we introduce the Fr\'echet differentiability in \cref{def:Frechet}. As in \cref{fig:GD}, we introduce positive indices $\fr,\fq,\fk$ related to the domain regularity, data and transformation field, respectively.

\begin{definition}[Fr\'echet differentiability for DDs]\label{def:Frechet}
Given Banach vector spaces $\fX,\fY,\fZ$ and $\mfT \subseteq \fX$ an open subset of $\fX$, we set $\U^\textup{inc}(\fq)$ and $D$. Then, we consider $\btheta \in \mfT$ associated to a domain $D_\theta(\fr)$. Consequently, we introduce a quantity of interest $\U_\theta (\fr,\fq)\in Z$.
For a perturbation $\bh(\fk) \in \fY $ one seeks whether the Fr\'echet derivative at $\btheta=0$ along direction $\bh$:
\begin{align*}
\D\U : \mfT& \to \mB(\fY,\fZ) \\
\btheta & \mapsto  \D \U_\theta :=\bh \mapsto \langle \D \U_\theta,\bh\rangle = \lim_{\| \bh \|_\fY\to 0}\frac{\| \U_\theta(\bx + \bh) - \U_\theta(\bx)\|_\fZ}{\|\bh\|_\fY},
\end{align*}
if well-defined.
\end{definition}
Notice the parameter dependence of $\D \U_\theta (\fr,\fk,\fq)$.

\begin{remark} In some applications such as domain uncertainty quantification for computational electromagnetics \cite{aylwin2020domain}, it is judicious to embed the domain $D$ itself --and the transformation $T$-- into Banach spaces by defining the space of admissible domains. A simple manner is the following:
\be
\mathfrak{G}:= \{ T:  T(\Gamma)\in C^{r,1}\}.
\ee
Therefore, one can introduce the following $\eps$-neighborhood for $\eps >0$:
\be
\mathfrak{G}_{\eps} := \{ \tilde{T} \in W^{k+1,\infty}(\Gamma) : \exists T \in \mathfrak{T}, ~ \|\tilde{T} - T \|_{W^{k+1,\infty} (\Gamma)} \leq \eps \}.
\ee
The latter allows to prove Fr\'echet differentiability in any $T\in \mathfrak{G}_\eps$. Remark the connection between differentiability in $T$ and in $\theta=0$.
Finally, we can extend all these notions to perturbations in the complex plane, by replacing $W^{k+1,\infty}(\Gamma)\equiv W^{k+1,\infty}(\Gamma ; \IR^d)$ \cite{cohen2016shape}. Notice that embeddings in this subsection can be set up at ``G\^ateaux level", by keeping the directional transformation instead.
\end{remark}
Back to \cref{sec:main_results}, we have now proved that for $\fC$ a $C^{\fr,1}$-domain, $\fr,\fq$ and $\fs:=\dot{\fs}$ (resp.~$\fs:=\tilde{\fs}$) such as is \cref{thm:MD} (resp.~\cref{thm:SD}), the MD (resp.~SD) belongs to $\fZ:= \fH^{\fs+1}(\fD)$ and the CMD (resp.~CSD) belongs to $\fZ_\Gamma:=\IH^\fs(\Gamma)$. Finally, the norms in theses spaces are wavenumber explicit and exhibit the same rates as $C$ in \cref{thm:MD,thm:SD}.
\subsection{Continuous data dependence}
Next, continuous dependence on the data in \cref{coro:sharp_MD} and \cref{coro:sharp_SD} provides direct boundedness estimates for the DDs function to $\bv$, ensuring directional differentiability. Finally, we embed the domain transformation and replace the directional transformation $\bv$ by a Fr\'echet transformation $\btheta$, embedded in $\fY:=W^{\fk+1,\infty}(\fD)$. Therefore, (S1) applies with $\btheta$ instead of $\bv$ and we deduce directly that all quantities in \cref{sec:main_results} are Fr\'echet differentiable in $\btheta \in \fX :=W^{\fr+1,\infty}(\fD)$ as a mapping from $\fY \to \fZ$ (resp.~$\fY \to \fZ_\Gamma$ for the CDDs). Last point is the continuity of the Fr\'echet derivative in $\fY$. 
{}
To begin with, $\U_\theta \circ T$ is Fr\'echet differentiable, hence continuous, providing $\|\U_\theta \circ T - \U \|_\fZ \to 0$ for $\btheta \to 0$. Next, as $\dot{\U}$ and $\dot{\U_\theta}$ are solution of a BVP:
\be 
\|\dot{\U}_\theta- \dot{\U}  \|_{\fZ} \leq C \| \U_\theta\circ T - \U \|_\fZ \|\bh\|_{\fY} \to 0 \textup{ as } \btheta \to 0. 
\ee
The same proof applies to the CDD in $\fZ_\Gamma$. Finally, concerning the SD, we introduce all appropriate extensions, and observe that $\U_\theta$ is continuous at a neighborhood of $\btheta = 0 \in \fY$. Next, provided that $\U_\theta'$ and $\U'$ solve BVPs, their difference:
\be 
\|\tilde{\U}'_\theta - \U'\|_{\fZ} \leq C \| \tilde{\U}_\theta - \U \|_\fZ \|\bh\|_\fY  \to 0 \textup{ as } \btheta \to 0.
\ee
Therefore, application of the mean value theorem \cite{djellouli1999continuous} provides the final stability estimates. For the sake of understanding, we sum up in  \cref{tab:Final} the equivalence between $\fC,\fX,\fY,\fZ,\fZ_\Gamma$ such as defined in this section, and the spaces they relate to in \cref{thm:MD,thm:SD}.
\begin{table}[t]
\begin{center}
\renewcommand\arraystretch{1.4}
\resizebox{13.4cm}{!} {
\begin{tabular}{|>{\centering\arraybackslash}m{.8cm}|
|>{\centering\arraybackslash}m{.8cm}
    |>{\centering\arraybackslash}m{1.8cm}
    |>{\centering\arraybackslash}m{1.8cm}
    |>{\centering\arraybackslash}m{1.3cm}
    |>{\centering\arraybackslash}m{1cm}|
    |>{\centering\arraybackslash}m{0.8cm}
    |>{\centering\arraybackslash}m{1.2cm}
    |>{\centering\arraybackslash}m{1.4cm}|
    |>{\centering\arraybackslash}m{1.6cm}
    |>{\centering\arraybackslash}m{4cm}|
    } \hline   
  \multirow{2}{*}{Case} & \multirow{2}{*}{$\fC$} & \multirow{2}{*}{$\fX$} &   \multirow{2}{*}{$\fY$} & \multirow{2}{*}{$\fZ$}  & \multirow{2}{*}{$\fZ_\Gamma$}& \multirow{2}{*}{$\fr$}  & \multirow{2}{*}{$\fq$}  & \multirow{2}{*}{$\fk$} & \multirow{2}{*}{$\beta$} & \multirow{2}{*}{$\fs$}  \\ 
  &&&&& && &  & &\\ \hline\hline
  \multirow{2}{*}{MD}  & \multirow{2}{*}{$C^{\fr,1}$}& \multirow{2}{*}{$W^{\fr+1,\infty}(D)$}  & \multirow{2}{*}{$W^{\fk+1,\infty}(D)$} & \multirow{2}{*}{$\fH^{\fs+1}(\fD)$} & \multirow{2}{*}{$\IH^{\fs}(\Gamma)$} & \multirow{2}{*}{$\IN_1$}  & \multirow{2}{*}{$\geq 1$} & \multirow{2}{*}{$1\leq \fk\leq \fr $}  &  $\beta = 0$ & $\min\left(\fq,\fk\right)$ \\ \cline{10-11}
& &&&& && && $\beta=1,2,3$ & $\min\left(\fq,\fr-\half,\fk\right)$\\ \hline \hline
  \multirow{2}{*}{SD}  & \multirow{2}{*}{$C^{\fr,1}$} & \multirow{2}{*}{$W^{\fr+1,\infty}(\Gamma)$} & \multirow{2}{*}{$W^{\fk+1,\infty}(\Gamma)$} & \multirow{2}{*}{$\fH^{\fs+1}(\fD)$} & \multirow{2}{*}{$\IH^{\fs}(\Gamma)$} & \multirow{2}{*}{$\IN_1$}  & \multirow{2}{*}{$\geq 0$} & \multirow{2}{*}{$\leq \fr $}  &\multirow{2}{*}{$-$} & \multirow{2}{*}{$\min \left(\fq-1, \fr-\half, \fk+\half\right) $}  \\
& &&& && && && \\ \hline\end{tabular}}
\end{center} 
\caption{Summary of the equivalence between function spaces in (S2) and in \cref{sec:main_results}.}
\label{tab:Final}
\end{table}

\section{Concluding Remarks}
\label{sec:concl}
In this work, we tackled a family of Helmholtz scattering problems and proved new differentiability results concerning the domain-to-solution and domain-to-Cauchy data mapping via explicit shape calculus, the latter relying on the characterization of the DDs as a BVP and on adapted functional analysis. This framework allows to tackle Lipschitz domains, domains with corners and wavenumber analysis. Those will be investigated elsewhere.

Our contribution is key to prove finer Fr\'echet differentiability \cite{Frechet,costabel2012shapeO,costabel2012shape} or shape holomorphy results \cite{hiptmair2018large,cohen2016shape,jerez2017electromagnetic}. Besides, it paves the way towards a framework allowing to deal with general BVPs and their boundary reduction. Moreover, the formulas for the DDs of Cauchy data allow for numerous applications in first-order shape boundary methods \cite{escapil2020helmholtz}. Recent works \cite{hiptmair2013shape,hiptmair2017shape} characterized the shape derivative for a variety of scattering problems appearing in acoustics and electromagnetism. These results rely on exterior calculus and differential forms, which could be an elegant manner to extend the results proposed in this work.

\section*{Acknowledgments}
This work was supported in part by Fondecyt Regular 1171491.

\bibliographystyle{siamplain}
\bibliography{references}
\end{document}